\documentclass[12pt,oneside,english]{amsart}
\usepackage[T1]{fontenc}
\usepackage[latin1]{inputenc}
\usepackage{indentfirst}
\usepackage{amsmath}
\usepackage{yhmath} 
\usepackage{amsfonts}
\usepackage{amssymb}
\usepackage{textcomp} 
\usepackage{epic}
\usepackage{verbatim} 
\usepackage[english]{babel}

\theoremstyle{definition}
\newtheorem{defn}{{Definition}}
\newtheorem*{defn*}{{Definition}}

\newtheorem*{Q*}{{Question}}

\theoremstyle{plain}
\newtheorem{thm}[defn]{Theorem}
\newtheorem*{thm*}{Theorem}

\newtheorem*{cor*}{Corollaire}

\newtheorem{prop}[defn]{{Proposition}}
\newtheorem*{prop*}{{Proposition}}
\newtheorem{lem}[defn]{{Lemma}}

\theoremstyle{remark}
\newtheorem*{rmq}{{Remark}}

\title{A remark on Cantor derivative}
\author{C\'e{d}ric Milliet}
\curraddr[C. Milliet]{Universit\'e de Lyon, Universit\'e Lyon 1, Institut Camille Jordan UMR 5208 CNRS,
43 boulevard du 11 novembre 1918,\newline%
\indent 69622 Villeurbanne Cedex, France}%
\email[C. Milliet]{milliet@math.univ-lyon1.fr}%

\begin{document}

\begin{abstract}It is shown that, modulo an equivalence relation induced by finite correspondences preserving Cantor rank, the class of topological spaces is an integral semi-ring on which the Cantor derivative is precisely a derivation.\end{abstract}

\maketitle

The notions of limit point and derivated space have both been introduced by Georg Cantor in~1872 to derivate sets of convergence of trigonometric series. In \cite{Can} Cantor shows that the representation of a function as a trigonometric series is unique on a set minus some finite Cantor ranked set. It was as he considered points systems having infinite Cantor rank that he introduced transfinite induction. That was in~1880, three years before his set theory. Cantor's words are \emph{Grenzpunkt} and \emph{abgeleitete Punktmenge} for "limit point" and "derived space" respectively. He was already writing $P'$ or $P^{(1)}$ for the first derivative of a set of points~$P$. However, it does not seem at all that Cantor had in mind a Leibniz formula, but it is intriguing that the class of topological spaces can naturally be turned into a semi-ring where Cantor's derivative is actually a derivation.

All spaces considered in the sequel are topological spaces. A \emph{correspondence} between two spaces $X$ and $Y$ is any relation $R\subset X\times Y$ such that the projections to $X$ and $Y$ are onto. We write $R^{-1}$ for the inverse correspondence of $R$ from $Y$ to $X$. If $O$ is an open set in $X$, we define $R(O)$ as $\{y\in Y: (x,y)\in R \text{ for some }x \in O\}$. The relation $R$ is \emph{continuous} if for every open set $O$ in $Y$, the set $R^{-1}(O)$ is open in $X$. It is \emph{open} if $R^{-1}$ is continuous. It is a \emph{$n$-to-$m$} correspondence if for all $x,y$ in $X\times Y$, the set $R(\{x\})$ has cardinal at most $m$ and $R^{-1}(\{y\})$ has cardinal at most $n$. A correspondence is \emph{finite} if it is $n$-to-$m$ for some non-zero integers $n$ and $m$.

Let $X$ be any topological space. We slightly modify the usual definitions to avoid the use of any separation axiom, and call a point \emph{isolated} if it belongs to a finite open set. Otherwise, we say that it is a \emph{limit point}. We shall write $X'$ for the \emph{derivative} of $X$, that is, the set of limit points with the induced topology, and define a descending chain of closed subsets $X^\alpha$ by setting, inductively
\begin{itemize}
\item[]$X^0=X$
\item[]$X^{\alpha+1}=\big(X^\alpha\big)'\ \text{for a successor ordinal}$
\item[]$X^\lambda=\displaystyle\bigcap_{\alpha<\lambda}X^\alpha\ \text{for a limit ordinal}\ \lambda$\end{itemize}

We call \emph{Cantor-Bendixson rank} of $X$, written $CB(X)$, the least ordinal $\alpha$ such that $X^\alpha$ is empty if such an ordinal exists, or $\infty$ otherwise. The rank of a point $x$ is the supremum of the $\alpha$ such that $x\in X^\alpha$. A subset, or a point of $X$ has \emph{maximal} rank if it has the same Cantor-Bendixson rank as $X$. Otherwise, we say that it has \emph{small} rank.

We call a \emph{rough partition} of $X$, any covering of $X$ by open sets having maximal rank and small ranked intersections. We define the \emph{Cantor-Bendixson degree} $d(X)$ of $X$ to be the supremum cardinal of the rough partitions of $X$. Open continuous finite correspondences do preserve the rank~:

\begin{lem}\label{rangfibresfinies}Let $R$ be a $n$-to-$m$ correspondence between two spaces $X$ and $Y$.\begin{itemize} \item[$(i)$] If $R$ is open, $CB(X)\geq CB(Y)$.
\item[$(ii)$] If $R$ is continuous, $CB(X)\leq CB(Y)$.
\item[$(iii)$] If $R$ is continuous and open, then $X$ and $Y$ have the same rank and $$\frac{1}{m}\cdot d(Y)\leq d(X)\leq n\cdot d(Y)$$\end{itemize}\end{lem}

\begin{proof}$(i)$ Let $y$ be a limit point in $Y$ and $x$ in $R^{-1}(\{y\})$. For every neighbourhood $O$ of $x$, the image $R(O)$ is an infinite neighbourhood of $y$, so $O$ is infinite. Hence $R^{-1}(Y')\subset X'$. Inductively, one can prove that $R^{-1}(Y^\alpha)\subset X^\alpha$. This shows that $Y^\alpha\subset R(X^\alpha)$.

$(ii)$ $R$ is continuous if and only if $R^{-1}$ is open, and the result follows from $(i)$.

$(iii)$ If $R$ is a $n$-to-$m$ correspondence from $X$ to $Y$, then $R^{-1}$ is a $m$-to-$n$ correspondence from $Y$ to $X$ so it is sufficient to prove the second equality. We may assume that the degree of $Y$ is an integer $d$. In that case, the rank of $Y$ is a successor ordinal, say $\alpha+1$. By the previous points, $X$ also has rank $\alpha+1$. Let us suppose for a contradiction that there be $O_0,\dots,O_{d\cdot n}$ a sequence of $d\cdot n+1$ open sets in $X$ with maximal rank, and small intersections. The sets $O_i^\alpha$ are disjoint. As $R$ is $n$-to-something, for every subset $I$ of $[0,d\cdot n]$ having at least $n+1$ points, the intersection $\bigcap_{i\in I} R(O_i^\alpha)$ is empty, so there exist $d+1$ disjoint subsets $I_0,\dots,I_d$ of $[0,d\cdot n]$ such that for all $j$, the set $\bigcap_{i\in I_j} R(O_i^\alpha)$ is nonempty, and $I_j$ is maximal with this property. Let us write $V_j$ for $\bigcap_{i\in I_j} R(O_i)$. Every $V_j$ is an open set in $Y$, with the same rank as $Y$ by point $(ii)$, and $V_j\cap V_k$ has small rank for $k\neq j$ in $[0,d]$, a contradiction with $Y$ having degree~$d$.\end{proof}

\begin{rmq}In Model Theory, Cantor-Bendixson rank gave birth to Morley rank in omega-stable theories \cite{Mor}. Points $(i)$ and $(ii)$ are well known by logicians and indicate that finite-to-one definable maps are "valuable" arrows for preserving a good notion of dimension~\cite{PoiPi}.\end{rmq}

Let $X$ be a set and $f$ a map on $2^X$ also defined on $2^X\times 2^X$. We say that $f$ is \emph{multiplicative} if $f(A\times B)$ equals $f(A)\times f(B)$. We call $f$ a \emph{pre-derivation} if $f(A\times B)$ equals $f(A)\times B\cup A\times f(B)$. Note a duality between some multiplicative maps and pre-derivations~:

\begin{lem}\label{der}Let $X$ be a set, and $f$ a map from $2^X$ to $2^X$ such that $f(A)\subset A$ for every subset $A$ of $X$. Write $\bar f(A)$ for $A\setminus f(A)$. Suppose in addition that $f$ is defined on finite cartesian products of $2^X$, and that $f$ is multiplicative. Then $\bar f$ is a pre-derivation.\end{lem}

For two spaces $X$ and $Y$, we shall write $X \simeq Y$ if there is a finite correspondence from $X$ to $Y$ preserving the Cantor rank of each point. This is an equivalence relation.

We denote by $X\coprod Y$ the \emph{topological disjoint union} of $X$ and $Y$, that is, their disjoint union together with the finest topology for which the canonical injections $X\rightarrow X\coprod Y$ and $Y\rightarrow X\coprod Y$ are continuous.

Recall that any ordinal can be uniquely written as $\omega^{\alpha_1}. n_1+\hdots+\omega^{\alpha_k}. n_k$ where $\alpha_1,\hdots,\alpha_k$ is a strictly decreasing chain of ordinals and $n_1,\hdots ,n_k$ are non-zero integers. This is known as its \emph{Cantor normal form}. If $\alpha$ and $\beta$ are two ordinals with normal forms $\omega^{\alpha_1}. m_1+\hdots+\omega^{\alpha_k}. m_k$ and $\omega^{\alpha_1}. n_1+\hdots+\omega^{\alpha_k}. n_k$ respectively (with zero integers possibly to make their length match), their \emph{Cantor sum} $\alpha\oplus\beta$ is the ordinal $\omega^{\alpha_1}. (m_1+n_1)+\hdots+\omega^{\alpha_k}. (m_k+n_k)$.

\begin{prop}The class of topological spaces modulo $\simeq$, together with $\coprod$ and $\times$ is a commutative integral semi-ring, on which Cantor's derivative is a derivation. The Cantor-Bendixson rank is a homomorphism from the class of compact spaces modulo $\simeq$ to the ordinals, preserving the structure of ordered semi-ring. Here, the ordinals are considered with the operations $max$, $\oplus$, and their natural ordering.\end{prop}

\begin{proof}It is not difficult to check that $\coprod$ and $\times$ survive modulo $\simeq$, are still associative, and that $\times$ is still distributive over $\coprod$. Note that if $X$ and $Y$ are two closed sub-spaces of some topological space $Z$, then $(X\cup Y)'=X'\cup Y'$. It follows that for any pair $X$ and $Y$ of topological spaces, $(X\coprod Y)'$ is homeomorphic to $X'\coprod Y'$, so Cantor's derivative preserves the sum $\coprod$ modulo $\simeq$. If we write $X^{isol.}$ for the set of isolated points in $X$, note that $(X\times Y)^{isol.}$ equals $X^{isol.}\times Y^{isol.}$. So~Cantor's derivative is a pre-derivation by Lemma~\ref{der}. The canonical map $f:X'\times Y\coprod X\times Y'\rightarrow X'\times Y\cup X\times Y'$ is a two-to-one continuous correspondence, so $f(x)$ has greater rank that $x$ for every $x$ by Lemma \ref{rangfibresfinies}. For the converse, as $X'\times Y$ and $X\times Y'$ are both closed in $X\times Y$, one has $(X'\times Y\cup X\times Y')'= (X'\times Y)'\cup (X\times Y')'$, so $f$ preserves the rank.


Cantor-Bendixson rank is well defined on the class of a topological space modulo $\simeq$. Clearly, the rank of a sum equals the maximum of the ranks. By induction, for any topological spaces $X$ and $Y$, we get $(X\times Y)^\alpha=\cup_{\beta\oplus\gamma=\alpha}A^\beta\times B^\gamma$. Note that if $X$ is a compact space, $CB(X)$ is a successor ordinal, the predecessor of which we write $CB^*(X)$. For two compact spaces $X$ and $Y$, this shows that $CB^*(X\times Y)$ equals $CB^*(X)\oplus CB^*(Y)$.
\end{proof}

Lemma \ref{rangfibresfinies} implies that two spaces in finite continuous open correspondence have the same Cantor rank. Reciprocally, this invariant classifies countable Hausdorff locally compact spaces up to finite continuous open correspondances. This a consequence of the following~:


\begin{thm}[Mazurkiewicz-Sierpi\'nski \cite{MS}]Every countable compact Hausdorff space is homeomorphic to some well-ordered set with the order topology.\end{thm}

\begin{proof}We give a short proof of a slightly more general result : we show that two countable locally compact Hausdorff spaces $X$ and $Y$ of same Cantor-Bendixson rank and degree are homeomorphic (in particular homeomorphic to $\omega^\alpha.d+1$ if they are compact of rank $\alpha+1$ and degree $d$).

Suppose first that $X$ and $Y$ be compact of rank $\alpha+1$.
Note that they are the disjoint union of finitely many compact spaces of degree $1$, so one may assume that their degree is $1$. We build a homeomorphism from $X$ to $Y$ by induction on the rank.
Let $X_1,X_2,\dots$ and $Y_1,Y_2,\dots$ be two sequences of clopen sets roughly partitioning $X\setminus X^\alpha$ and $Y\setminus Y^\alpha$ respectively. As $X_1$ has smaller rank or degree than some finite union of $Y_i$, we may assume that $X_1$ has smaller rank or degree than $Y_1$, and that $Y_1$ has smaller rank or degree that $X_2$ etc. We then build a back and forth : by induction hypothesis, there is sequence $f_1,g^{-1}_1,f_2,g^{-1}_2\dots$ of homeomorphism respectively from $X_1$ to some clopen $\tilde Y_1\subset Y_1$, from $Y_1\setminus \tilde Y_1$ to some clopen set $\tilde X_2\subset X_2$, from $X_2\setminus \tilde X_2$ to $\tilde Y_3\subset Y_3$ etc. We call $f$ be the union of all $f_i$ and $g_i$, union one more map $f_\omega$ from $X^\alpha$ to $Y^\alpha$ and show that $f$ is continuous. We may show sequential continuity as the spaces are metrisable. If $x_i$ is a sequence of limit $x$ in $X$, either $x$ has small rank and belongs to some $X_j$, so $f(x_i)$ has limit $f(x)$ by continuity of $f_j$ and $g_j$. Or $x$ has maximal rank. If $b$ is an accumulation point of the sequence $f(x_i)$ having small rank, it belongs to some clopen set $Y_j$, so the compact set $X_j$ contains infinitely many $x_i$, a contradiction. So the sequence $f(x_i)$ has only one accumulation point and must converge to $y$.



If the spaces are locally compact, one can write them as a countable union of increasing clopen compact spaces, and builds a back and forth similarly.\end{proof}


\end{document}